\documentclass[12pt,reqno]{amsart}
\usepackage{amsfonts}
\usepackage{amssymb}
\numberwithin{equation}{section}
\theoremstyle{plain}
 \newtheorem{thm}{Theorem}[section]

 \newtheorem{prop}[thm]{Proposition}
\theoremstyle{definition}
 \newtheorem{defn}[thm]{Definition}
 \newtheorem{ex}[thm]{Example}
 
\newcommand{\al}{\alpha}
\newcommand{\bt}{\beta}
\newcommand{\gm}{\gamma}
\newcommand{\Gm}{\Gamma}
\newcommand{\dl}{\delta}

\newcommand{\ep}{\varepsilon}

\newcommand{\et}{\eta}

\newcommand{\ld}{\lambda}
\newcommand{\Ld}{\Lambda}
\newcommand{\sg}{\sigma}

\newcommand{\ph}{\varphi}
\newcommand{\Ph}{\Phi}

\newcommand{\Ps}{\Psi}
\newcommand{\rh}{\rho}

\newcommand{\dar}{\downarrow}
\newcommand{\uar}{\uparrow}

\newcommand{\mcal}{\mathcal}
\newcommand{\mrm}{\mathrm}
\newcommand{\mfr}{\mathfrak}

\newcommand{\wh}{\widehat}

\newcommand{\la}{\langle}
\newcommand{\ra}{\rangle}
\newcommand{\ts}{\textstyle}

\newcommand{\C}{\mathbb{C}}
\newcommand{\R}{\mathbb{R}}

\newcommand{\les}{\leqslant}
\newcommand{\ges}{\geqslant}
\newcommand{\law}{\mathcal L}
\newcommand{\cl}{\colon}

\newcommand{\Rpl}{\mathbb{R}_{+}^{\circ}}
\renewcommand{\labelenumi}{(\roman{enumi})}

\begin{document}
\setlength{\baselineskip}{16pt}
\setlength{\parindent}{1.8pc}
\allowdisplaybreaks
\hyphenation{self-de-com-pos-able self-decom-pos-abil-ity de-com-pos-abil-ity}

{\bf
{\large
\noindent
Weak drifts of infinitely divisible distributions\\
and their applications}

\vspace{5mm}
\noindent
Ken-iti Sato\footnote{Hachiman-yama 1101-5-103, Tenpaku-ku, Nagoya, 
468-0074 Japan.} \footnote{To whom correspondence should be addressed. 
E-mail: ken-iti.sato@nifty.ne.jp} and 
Yohei Ueda\footnote{Department of Mathematics, Keio University, 3-14-1, Hiyoshi, 
Kohoku-ku, Yokohama 223-8522, Japan}}

\vspace{5mm}
\begin{quotation}
{\small
Weak drift of an infinitely divisible distribution
$\mu$ on $\R^d$ is defined by analogy with weak mean; 
properties and applications of weak drift are given.  
When $\mu$ has
no Gaussian part, the weak drift of $\mu$ equals the minus of the weak mean of the
inversion $\mu'$ of $\mu$.  Applying the concepts of having weak drift $0$
and of having weak drift $0$ absolutely, the ranges, the absolute ranges, 
and the limit of the
ranges of iterations are described for some stochastic integral mappings.
For L\'evy processes the concepts of weak mean and weak drift 
are helpful in giving necessary and sufficient conditions 
for the weak law of large numbers and for the weak version of Shtatland's theorem
on the behavior near $t=0$; those conditions are obtained from each other
through inversion.

}
\end{quotation}

\section{Introduction}

This paper introduces the notion of weak drift of an infinitely 
divisible distribution
and applies it, first, to the relations between inversions of infinitely divisible 
distributions and conjugates of stochastic integral mappings studied in Sato \cite{S11b} and,
second, to the weak law of large numbers for L\'evy processes and the weak version of 
Shtatland's theorem on the behavior of L\'evy processes near $t=0$.

The basic notions in this paper are as follows. Let $ID=ID(\R^d)$ be the class of
infinitely divisible distributions on $\R^d$.  
The L\'evy--Khintchine triplet $(A_{\mu},\nu_{\mu},\gm_{\mu})$ of $\mu\in ID$ 
consisting of the Gaussian covariance matrix $A_{\mu}$, the L\'evy measure
$\nu_{\mu}$, and the location parameter $\gm_\mu$ is given by the formula
\begin{equation}\label{1.1}
\wh\mu(z)=\exp\Bigl[ -\tfrac12 \langle z,A_{\mu} z\rangle +\int_{\mathbb{R}^d}
(e^{i\langle z,x\rangle}-1-i\langle z,x\rangle
1_{\{|x|\les 1\}}(x) ) \nu_{\mu}(dx)
+i\langle \gamma_{\mu},z\rangle\Bigr]
\end{equation}
for the characteristic function $\wh\mu(z)$, $z\in\R^d$, of $\mu$.
Recall that $\nu_{\mu}(\{ 0\})=0$ and $\int_{\R^d} (|x|^2\land 1)\nu_{\mu}(dx)
<\infty$.  If $\int_{|x|\les1}|x|\nu_{\mu}(dx)<\infty$, then
\[
\wh\mu(z)=\exp\Bigl[-\tfrac12 \langle z,A_{\mu} z\rangle +\int_{\mathbb{R}^d}
(e^{i\langle z,x\rangle}-1) \nu_{\mu}(dx)+i\langle \gamma_{\mu}^0,z\rangle
\Bigr],
\]
where $\gm_{\mu}^0$ is called the drift of $\mu$.  If $\int_{|x|>1}
|x|\nu_{\mu}(dx)<\infty$, then 
\[
\wh\mu(z)=\exp\Bigl[-\tfrac12 \langle z,A_{\mu} z\rangle +\int_{\mathbb{R}^d}
(e^{i\langle z,x\rangle}-1-i\langle z,x\rangle) \nu_{\mu}(dx)+i\langle 
m_{\mu},z\rangle\Bigr],
\]
where $m_{\mu}$ is the mean of $\mu$.
Let $ID_0=\{\mu\in ID\colon A_{\mu}=0\}$. 
For $\mu\in ID_0$ the inversion $\mu'\in ID_0$ of $\mu$ is defined as
$\nu_{\mu'}(B)=\int_{\R^d\setminus \{0\}} 1_B(|x|^{-2} x)|x|^2 \nu_{\mu}
(dx)$ for all $B$ in $\mcal B(\R^d)$, the class of Borel sets in $\R^d$,
and $\gm_{\mu'}=-\gm_{\mu}+\int_{|x|=1} x\nu_{\mu}(dx)$.
Any $\mu\in ID_0$ has its inversion $\mu'\in ID_0$ and we have $\mu''=\mu$.
The inversion $\mu'$ has drift $\gm_{\mu'}^0$ if and only if $\mu$ has
mean $m_{\mu}$; we have $\gm_{\mu'}^0=-m_{\mu}$.
Many other properties of the inversion are given in \cite{S07,S11b}.
For example, for $0<\al<2$, $\mu$ is $\al$-stable if and only if 
$\mu'$ is $(2-\al)$-stable, and
$\mu$ is strictly $\al$-stable if and only if 
$\mu'$ is strictly $(2-\al)$-stable.

We use the following notation throughout this paper:  $\mfr C_0=\mfr C\cap ID_0$ 
for any $\mfr C\subset ID$;  $\mfr C'=\{\mu'\colon \mu\in\mfr C\}$ for 
any $\mfr C\subset ID_0$.

The notions that a distribution $\mu\in ID$ has weak mean $m_{\mu}$
and that $\mu\in ID$ has weak mean $m_{\mu}$ absolutely are introduced
in \cite{S10}.
In Section 2 of this paper we will recall those definitions 
and then define 
the notions that a distribution $\mu\in ID$ has weak drift $\gm_{\mu}^0$
and that $\mu\in ID$ has weak drift $\gm_{\mu}^0$ absolutely. Properties of weak means and
weak drifts are in parallel.  Moreover we will prove that $\mu'$ has 
weak drift $\gm_{\mu'}^0$ if and only if $\mu$ has weak
mean $m_{\mu}$ and that $\gm_{\mu'}^0=-m_{\mu}$.

Let $\{X_t^{(\rh)}\cl t\ges0\}$ be a L\'evy process on $\R^d$ such that 
$\law (X_1^{(\rh)})$, the distribution of $X_1^{(\rh)}$, equals $\rh$.
We consider improper stochastic integrals with respect to $\{X_t^{(\rh)}\}$
in two cases.
{\def\labelenumi{(\arabic{enumi})}
\begin{enumerate}
\item Let $0<c\les\infty$ and let $f(s)$ be a locally square-integrable function
on $[0,c)$.  We say that the improper stochastic integral 
$\int_0^{c-} f(s)dX_s^{(\rh)}$ is definable if  
$\int_0^q f(s)dX_s^{(\rh)}$ is convergent in probability as  
$q\uar c$.  Define the mapping $\Ph_f$ from $\rh$ to
$\Ph_f\,\rh=\law \bigl( \int_0^{c-} f(s)dX_s^{(\rh)} \bigr)$; its domain 
$\mfr D(\Ph_f)$ is the class of $\rh\in ID$ such that
$\int_0^{c-} f(s)dX_s^{(\rh)}$ is definable.
\item Let $0<c<\infty$ and let $f(s)$ be a locally square-integrable function
on $(0,c]$.  We say that the improper stochastic integral 
$\int_{0+}^c f(s)dX_s^{(\rh)}$ is definable if 
$\int_p^c f(s)dX_s^{(\rh)}$ is convergent in probability as   
$p\dar 0$.  Define the mapping $\Ph_f$ from $\rh$ to
$\Ph_f\,\rh=\law \bigl( \int_{0+}^c f(s)dX_s^{(\rh)} \bigr)$; its domain 
$\mfr D(\Ph_f)$ is the class of $\rh\in ID$ such that
$\int_{0+}^c f(s)dX_s^{(\rh)}$ is definable.
\end{enumerate}}
\noindent In any of the cases (1) and (2),  $\Ph_f$ is called 
a stochastic integral mapping.  
Its range $\mfr R(\Ph_f)=\{\Ph_f\rh\cl \rh\in\mfr D(\Ph_f)\}$
is a subclass of $ID$.
If $c<\infty$ and $\int_0^c f(s)^2 ds<\infty$, then
$\int_0^{c-} f(s)dX_s^{(\rh)}=\int_{0+}^c f(s)dX_s^{(\rh)}=\int_0^c
f(s)dX_s^{(\rh)}$ for all $\rh\in ID$.
If $f=f_h$ is defined from a
function $h$ in some way (see Section 3 for precise formulation),
the stochastic integral mapping $\Ph_f$ is denoted by $\Ld_h$ as in \cite{S11b}.
By the transformation of $h(u)$ to $h^*(u)=h(u^{-1})u^{-4}$ the conjugate 
of $\Ld_h$ is defined by 
$\Ld_{h^*}$ and denoted by $(\Ld_h)^*$ or $\Ld_h^*$. 
Thus $\Ld_h^*=\Ld_{h^*}=\Ph_{f_{h^*}}$ and $(\Ld_h^*)^*=\Ld_h$. 
The relations of $\mfr D(\Ld_h)$ and 
$\mfr R(\Ld_h)$ with $\mfr D(\Ld_h^*)$ and $\mfr R(\Ld_h^*)$   
are studied in \cite{S11b}.
It is closely connected with the inversion.  
Thus $\rh\in \mfr D(\Ld_h)_0$ and $\Ld_h\rh=\mu$ if and only if
$\rh'\in \mfr D(\Ld_h^*)_0$ and $\Ld_h^*\rh'=\mu'$.
In the description of $\mfr R(\Ld_h)$ and the range $\mfr R^0
(\Ld_h)$ of absolutely definable $\Ld_h$ (see Section 3 for definition), 
the conditions of having weak mean $0$ 
and of having weak mean $0$ absolutely are sometimes useful, as is shown 
in \cite{S10,S11b}. We will show in Section 3 that
the conditions of having weak drift $0$ and of having weak drift $0$ absolutely are
useful in the description of $\mfr R(\Ld_h^*)$ and $\mfr R^0(\Ld_h^*)$.
In Section 4 a similar fact will be shown in the description of 
$\mfr R_{\infty}(\Ld_h)$ and
$\mfr R_{\infty}(\Ld_h^*)$, the limit of the ranges of iterations of $\Ld_h$
and $\Ld_h^*$, respectively.

In Section 5, we will give a necessary and sufficient condition
for a L\'evy process $\{X_t^{(\mu)}\}$ on $\R^d$ to satisfy the weak law of 
large numbers as $t\to\infty$ is that $\mu$ has weak mean and satisfies 
the condition 
$\lim_{t\to\infty} t\int_{|x|>t} \nu_{\mu}(dx)=0$.  On the other hand we will
show, using the inversion, that a L\'evy process $\{X_t^{(\mu)}\}$ 
without Gaussian part 
satisfies
the weak version of Shtatland's theorem \cite{Sh65} (that is, $t^{-1}X_t^{(\mu)}$ 
converges in law to some constant as $t\dar 0$) if and only if the L\'evy 
process $\{X_t^{(\mu')}\}$ satisfies the weak law of 
large numbers.  Thus it will be  shown that a necessary and sufficient condition
for a L\'evy process $\{X_t^{(\mu)}\}$ without Gaussian part 
to satisfy the weak version of 
Shtatland's theorem is that $\mu$ has a weak drift and satisfies
$\lim_{\ep\dar 0} \ep^{-1}\int_{|x|\les\ep} |x|^2 \nu_{\mu}(dx)=0$.


\section{Weak drifts of infinitely divisible distributions}

We say that $\mu\in ID$ has weak mean in $\R^d$ if 
\begin{equation}\label{2.1}
\int_{1<|x|\les a} x\nu_{\mu}(dx)\text{ is convergent in $\R^d$ as $a\to\infty$}.
\end{equation}
We say that $\mu\in ID$ has weak mean $m_{\mu}$ if \eqref{2.1} holds and 
$\wh\mu(z)$ satisfies
\begin{equation}\label{2.1a}
\wh\mu(z)=\exp\Bigl[ -\tfrac12 \langle z,A_{\mu} z\rangle +\lim_{a\to\infty}
\int_{|x|\les a}(e^{i\langle z,x\rangle}-1-i\langle z,x\rangle) \nu_{\mu}(dx)
+i\langle m_{\mu},z\rangle\Bigr].
\end{equation}
If $\mu\in ID$ has mean $m_{\mu}$, then $\mu$ has weak mean $m_{\mu}$.
If $\mu\in ID$ has weak mean $m_{\mu}$, then $m_{\mu}=\gm_{\mu}+\lim_{a\to\infty}
\int_{1<|x|\les a} x\nu_{\mu}(dx)$. 
Let $(\bar\nu_{\mu}(dr),\ld_r^{\mu}(d\xi))$ be a spherical decomposition of 
$\nu_{\mu}$, that is,
\[
\nu_{\mu}(B)=\int_{\Rpl}\bar\nu_{\mu}(dr)\int_{S} 1_B(r\xi)\ld_r^{\mu}(d\xi),
\qquad B\in\mcal B(\R^d),
\]
where $\bar\nu_{\mu}$ is a $\sg$-finite measure on $\Rpl=(0,\infty)$ with
$\bar\nu_{\mu}(\Rpl)\ges 0$  and
$\{\ld_r^{\mu}\cl r\in\Rpl\}$ is a measurable family of $\sg$-finite measures
on $S=\{\xi\in\R^d\cl |\xi|=1\}$ with $\ld_r^{\mu}(S)>0$ 
($S$ is the unit sphere if $d\ges 2$ or the two-point set 
$\{1,-1\}$ if $d=1$); the decomposition is unique up to a change 
to $(c(r) \bar\nu_{\mu}(dr),
c(r)^{-1}\ld_r^{\mu}(d\xi))$ with a positive, finite, measurable function
$c(r)$ on $\Rpl$.   
We say that $\mu\in ID$ has weak
mean in $\R^d$ absolutely if
\begin{equation}\label{2.2}
\int_{(1,\infty)}r\bar\nu_{\mu}(dr)\left|\int_{S} \xi\ld_r^{\mu}(d\xi)\right|<\infty.
\end{equation}
We say that $\mu\in ID$ has weak mean $m_{\mu}$ absolutely if \eqref{2.2} holds
and $\mu$ has weak mean $m_{\mu}$.  These notions are introduced in \cite{S10}.

Now we give the following definitions.

\begin{defn}\label{2d1}
We say that $\mu\in ID$ has \emph{weak drift} in $\R^d$ if 
\begin{equation}\label{2.3}
\int_{\ep<|x|\les 1} x\nu_{\mu}(dx)\text{ is convergent in $\R^d$ as $\ep\dar 0$}.
\end{equation}
We say that $\mu\in ID$ has weak drift $\gm_{\mu}^0$ if \eqref{2.3} holds and 
\[
\wh\mu(z)=\exp\Bigl[ -\tfrac12 \langle z,A_{\mu} z\rangle +\lim_{\ep\dar 0}
\int_{|x|>\ep}(e^{i\langle z,x\rangle}-1) \nu_{\mu}(dx)
+i\langle \gm_{\mu}^0,z\rangle\Bigr].
\]
\end{defn}

Property \eqref{2.3} is equivalent to saying that, for each $z\in\R^d$,
$\int_{|x|>\ep} (e^{i\la z,x\ra} -1) \nu_{\mu}(dx)$ is convergent in $\C$ 
as $\ep\dar 0$.

\begin{defn}\label{2d2}
We say that $\mu\in ID$ has \emph{weak drift} in $\R^d$ \emph{absolutely} if 
\begin{equation}\label{2.4}
\int_{(0,1]}r\bar\nu_{\mu}(dr)\left|\int_{S} \xi\ld_r^{\mu}(d\xi)\right|<\infty,
\end{equation}
where $(\bar\nu_{\mu}(dr),\ld_r^{\mu}(d\xi))$ is a spherical decomposition of
$\nu_{\mu}$.  
We say that $\mu\in ID$ has weak drift $\gm_{\mu}^0$ absolutely if \eqref{2.4} holds and 
$\mu$ has weak drift $\gm_{\mu}^0$.
\end{defn}

We remark that, if \eqref{2.4} holds for some spherical decomposition of
$\nu_{\mu}$, then it holds for any spherical decomposition of
$\nu_{\mu}$. 

\begin{prop}\label{2p1}
If $\mu\in ID$ has weak drift $\gm_{\mu}^0$, then $\gm_{\mu}^0=\gm_{\mu}
-\lim_{\ep\dar0}\int_{\ep<|x|\les 1} x\nu_{\mu}(dx)$.
\end{prop}

\begin{proof}
Compare \eqref{1.1} and \eqref{2.1a}.
\end{proof}

The following result is basic in this paper.

\begin{thm}\label{2t1}
Let $\mu\in ID_0$.  Then the inversion $\mu'$ of $\mu$ has weak drift 
in $\R^d$ if and only if $\mu$ has weak mean in $\R^d$.
The inversion $\mu'$ has weak drift in $\R^d$ absolutely 
if and only if $\mu$ has weak mean in $\R^d$ absolutely.
If $\mu'$ has weak drift $\gm_{\mu'}^0$, then $\gm_{\mu'}^0=-m_{\mu}$, where
$m_{\mu}$ is the weak mean of $\mu$.
\end{thm}

Since $\mu''=\mu$, we can interchange \lq\lq weak drift" and \lq\lq weak mean" 
in the second and third sentences of the theorem.

\begin{proof}[Proof of Theorem \ref{2t1}]
We have
\[
\int_{\R^d} h(x)\nu_{\mu'}(dx)=\int_{\R^d} h(|x|^{-2}x)|x|^2\nu_{\mu}(dx)
\]
for any $\R^d$-valued function $h(x)$ on $\R^d$ satisfying 
$\int |h(x)|\nu_{\mu'}(dx)=\int |h(|x|^{-2}x)|\,|x|^2\break \nu_{\mu}(dx)<\infty$,
as in the proof of Proposition 2.1 of \cite{S11b}.  Hence
\begin{equation}\label{2t1.1}
\int_{\ep<|x|\les 1} x\nu_{\mu'}(dx)=\int_{1\les|x|<1/\ep} x\nu_{\mu}(dx).
\end{equation}
Thus the second sentence of the theorem follows.  The fourth sentence also
follows, since
\begin{align*}
\gm_{\mu'}^0&=\gm_{\mu'}-\lim_{\ep\dar0} \int_{\ep<|x|\les 1} x\nu_{\mu'}(dx)
=-\gm_{\mu}+\int_{|x|=1} x\nu_{\mu}(dx)-\lim_{a\to\infty} \int_{1\les |x|<a} x\nu_{\mu}(dx)\\
&=-\gm_{\mu}-\lim_{a\to\infty} \int_{1< |x|\les a} x\nu_{\mu}(dx)=-m_{\mu}.
\end{align*}
Let $(\bar\nu_{\mu}(dr),\ld_r^{\mu}(d\xi))$ be a spherical decomposition of
$\nu_{\mu}$.  Define 
\[
\bar\nu^{\sharp}(E)=\int_{\Rpl} 1_E (r^{-1})r^2\bar\nu_{\mu}(dr),\qquad
E\in\mcal B(\Rpl).
\]
Then
\[
\nu_{\mu'}(B)=\int_{\Rpl}\bar\nu_{\mu}(dr)\int_S 1_B(r^{-1}\xi) r^2\ld_r^{\mu}(d\xi)
=\int_{\Rpl}\bar\nu^{\sharp}(dr) \int_S 1_B (r\xi)\ld_{r^{-1}}^{\mu}(d\xi).
\]
Hence $\nu_{\mu'}$ has a spherical decomposition $(\bar\nu_{\mu'}(dr),\ld_r^{\mu'}(d\xi))$
with $\bar\nu_{\mu'}=\bar\nu^{\sharp}$ and $\ld_{r}^{\mu'}=\ld_{r^{-1}}^{\mu}$.
It follows that
\begin{align*}
\int_{(0,1]}r\bar\nu_{\mu'}(dr)\left|\int_{S} \xi\ld_r^{\mu'}(d\xi)\right|
&=\int_{(0,1]}r\bar\nu^{\sharp}(dr)\left|\int_{S} \xi\ld_{r^{-1}}^{\mu}(d\xi)\right|\\
&=\int_{[1,\infty)}r^{-1}r^2\bar\nu_{\mu}(dr)\left|\int_{S} \xi\ld_r^{\mu}(d\xi)\right|,
\end{align*}
which yields the third sentence of the theorem.
\end{proof}

\begin{prop}\label{2p2}
Let $\mu\in ID$. If $\mu$ is symmetric, then $\mu$ has weak drift\/ $0$ absolutely.
\end{prop}

\begin{proof}
Assume that $\mu$ is symmetric, that is, $\mu(-B)=\mu(B)$ for $B\in\mcal B(\R^d)$.
Then $\wh\mu(z)$ is real. Hence $\nu_{\mu}$ is symmetric and $\gm_{\mu}=0$.
Thus $\int_{\ep<|x|\les 1} x\nu_{\mu}(dx)=0$. Hence $\mu$ has weak drift
and it follows from Proposition \ref{2p1} that the weak drift is $0$.
Let $(\bar\nu_{\mu}(dr),\ld_r^{\mu}(d\xi))$ be a spherical decomposition of 
$\nu_{\mu}$.  The symmetry of $\nu_{\mu}$ yields that, for 
$\bar\nu_{\mu}$-a.\,e.\ $r$, $\ld_r^{\mu}$ is symmetric, so that 
$\int_S\xi\ld_r^{\mu}(d\xi)=0$.  Hence \eqref{2.4} holds.
\end{proof}

\medskip
As a digression we mention a property of weak drift, which is applicable
to characterization of strict $1$-stability.
We say that the L\'evy measure 
$\nu_{\mu}$ of $\mu\in ID$ is of polar product type
if there are a finite measure $\ld_{\mu}$ on $S$ and a $\sg$-finite measure 
$\bar\nu_{\mu}$ on $\Rpl$ such that $\nu_{\mu}(B)=\int_S \ld_{\mu}(d\xi) \int_{\Rpl}
1_B(r\xi) \bar\nu_{\mu}(dr)$ for $B\in\mcal B(\R^d)$.

\begin{prop}\label{2p3}
Let $\mu\in ID$ with $\nu_{\mu}$ of polar product type. 
Assume that\linebreak $\int_{|x|\les1} |x|\nu_{\mu}(dx)=\infty$. 
Then the following five conditions are equivalent.
\begin{enumerate}
\item  $\mu$ has weak drift in $\R^d$.
\item  $\lim_{\ep\dar 0} \int_{\ep<|x|\les 1} x\nu_{\mu}(dx)=0$.
\item  $\mu$ has weak drift in $\R^d$ absolutely.
\item  $\mu$ has weak drift absolutely and\/ $\lim_{\ep\dar 0} \int_{\ep<|x|\les 1} 
x\nu_{\mu}(dx)=0$.
\item  $\ld_{\mu}$ in the definition of polar product type satisfies 
$\int_S\xi\ld_{\mu}(d\xi)=0$. 
\end{enumerate}
\end{prop}

\begin{proof}
The implications (iv) $\Rightarrow$ (ii) $\Rightarrow$ (i) and 
(iv) $\Rightarrow$ (iii) $\Rightarrow$ (i) are obvious.  We have
$\int_{\ep<|x|\les 1} x\nu_{\mu}(dx)=\int_S \xi\ld_{\mu}(d\xi) \int_{(\ep,1]}
r\bar\nu_{\mu}(dr)$, since $\ld_{\mu}$ is of polar product type. 
Moreover, since $\int_{|x|\les1} |x|\nu_{\mu}(dx)=\infty$,
we have $\int_{(0,1]}r\bar\nu_{\mu}(dr)=\infty$.
It follows that (i) implies (v).   
As $(\bar\nu_{\mu}(dr), \ld_{\mu}(d\xi))$ gives a spherical decomposition of 
$\nu_{\mu}$, (v) implies (iv). This proof
is similar to that of Proposition 3.15 of \cite{S10}. 
\end{proof}

\begin{ex}\label{2e1}
For $0<\al<2$ the L\'evy measure $\nu_{\mu}$ of an $\al$-stable distribution
$\mu$ on $\R^d$ is of polar product type in the form 
$\nu_{\mu}(B)=\int_S \ld_{\mu}(d\xi) \int_{\Rpl} 1_B(r\xi) r^{-1-\al}dr$. 
The condition $\int_{|x|\les1} |x|\nu_{\mu}(dx)=\infty$ is satisfied if
and only if $\mu$ is nontrivial (that is, not a $\dl$-measure) and $1\les \al<2$. 
Let $\mu$ be a 1-stable distribution on $\R^d$.  If $d=1$, then $\mu$ is
strictly $1$-stable if and only if $\nu_{\mu}$ is symmetric.   If $d\ges 2$,
then symmetry of $\nu_{\mu}$ implies strict $1$-stability of $\mu$, but
the L\'evy measure of strictly $1$-stable distribution is not always symmetric.
A necessary and sufficient condition for $\mu$ to be strictly $1$-stable is
that $\int_S\xi\ld_{\mu}(d\xi)=0$ (see \cite{S99}).  Hence 
Proposition  \ref{2p3} gives equivalent characterizations of strict 
$1$-stability for a nontrivial $1$-stable distribution.  Similar characterizations 
using weak mean are given in Example 3.16 of \cite{S10}. 
\end{ex}


\section{Ranges of conjugates of some stochastic integral mappings}

Conjugates of stochastic integral mappings are introduced in \cite{S11b} 
in the following way.  A function $h(u)$ is said to satisfy Condition (C) 
if there are $a_h$ and $b_h$ with $0\les a_h<b_h\les \infty$ such that $h$
is defined on $(a_h,b_h)$, positive, and measurable, and 
\[
\min\left\{ \int_{a_h}^{b_h} h(u) u^2 du,\;\int_{a_h}^{b_h} h(u)du
\right\}<\infty. 
\]
For any $h$ satisfying Condition (C) we define a function 
$h^*$ as $a_{h^*}=1/b_h$, $b_{h^*}=1/a_h$, and
\[
h^*(u)=h(u^{-1}) u^{-4}, \qquad u\in(a_{h^*}, b_{h^*}).
\]
Then $h^*$ automatically satisfies Condition (C) and we have $(h^*)^*=h$.

Let $h$ be a function satisfying Condition (C).  Define a strictly decreasing 
continuous function $g_h(t)$ as $g_h(t)=\int_t^{b_h} h(u)du$ for 
$t\in (a_h,b_h)$ and let $c_h=g_h(a_h +)$. Let $t=f_h(s)$, $0< s<c_h$,
be the inverse function of $s=g_h(t)$, $a_h<t< b_h$.
Then $f_h(s)$ is a strictly decreasing continuous function with $f_h(0+)=b_h$
and $f_h(c_h -)=a_h$.  For all $\rh\in ID$, 
the stochastic integral $\int_p^q f_h(s) dX_s^{(\rh)}$
with respect to a L\'evy process $\{ X_s^{(\rh)}\}$ with distribution $\rh$
at time $1$ is defined either for $0\les p<q<c_h=\infty$ or for 
$0<p<q\les c_h<\infty$.
If $h$ satisfies $\int_{a_h}^{b_h} h(u) u^2 du<\infty$, then
the stochastic integral mapping $\Ph_{f_h}$ is defined as 
$\Ph_{f_h}\rh=\law\left(\int_0^{c_h -} f_h(s)dX_s^{(\rh)}\right)$
whenever $\int_0^{c_h -} f_h(s)dX_s^{(\rh)}$ is definable.
If $h$ satisfies $\int_{a_h}^{b_h} h(u) du<\infty$, then $c_h<\infty$ and
$\Ph_{f_h}$ is defined as 
$\Ph_{f_h}\rh=\law\left(\int_{0+}^{c_h} f_h(s)dX_s^{(\rh)}\right)$
whenever $\int_{0+}^{c_h} f_h(s)dX_s^{(\rh)}$ is definable.
The mapping $\Ph_{f_h}$ is written as $\Ld_h$.   Given a function $h$ 
satisfying Condition (C), we call $\Ld_{h^*}$ 
the conjugate of $\Ld_h$ and write $\Ld_h^*=\Ld_{h^*}=\Ph_{f_{h^*}}$. 
Since $(h^*)^*=h$, the conjugate of $\Ld_h^*$ equals $\Ld_h$.
In the analysis of $\mfr D(\Ld_h)$ and $\mfr R(\Ld_h)$ we use the following
restriction and extension of $\Ld_h$.  
We say that $\Ld_h\rh$ is absolutely
definable if $\int_0^{c_h} |\log\wh\rh(f_h(s)z)|ds<\infty$ for $z\in\R^d$.
Let $\mfr D^0(\Ld_h)=\{ \rh\in \mfr D(\Ld_h)\colon \text{$\Ld_h\rh$ is 
absolutely definable}\}$ and $\mfr R^0(\Ld_h)=\{ \Ld_h\rh\colon \rh\in
\mfr D^0(\Ld_h) \}$.
If $h$ satisfies $\int_{a_h}^{b_h} h(u) u^2 du<\infty$, then
we say that $\Ld_h \rh$ is essentially definable if,
for some $\R^d$-valued function $k(q)$ for $0<q<c_h$ and some
$\R^d$-valued random variable $Y$, $\int_0^q f_h(s)dX_s^{(\rh)}-k(q)$ 
converges to $Y$ in probability as $q\uar c_h$.
If $h$ satisfies $\int_{a_h}^{b_h} h(u) du<\infty$, then $c_h<\infty$ and
we say that $\Ld_h \rh$ is essentially definable if,
for some $\R^d$-valued function $k(p)$ for $0<p<c_h$ and some
$\R^d$-valued random variable $Y$, $\int_p^{c_h} f_h(s)dX_s^{(\rh)}-k(p)$ 
converges to $Y$ in probability as $p\dar 0$.
Let $\mfr D^{\mrm{e}} (\Ld_h)=\{\rh\in ID\colon \text{$\Ld_h \rh$ is 
essentially definable} \}$ and let $\mfr R^{\mrm{e}} (\Ld_h)$ be the
class of $\mu=\law (Y)$ where all 
$\rh\in \mfr D^{\mrm{e}} (\Ld_h)$ and all $k$ and $Y$ that can be 
chosen in the definition of essential definability of $\Ld_h \rh$ are 
taken into account.
Notice that $\mfr D^0 (\Ld_h)\subset \mfr D (\Ld_h) \subset \mfr D^{\mrm{e}} (\Ld_h)$
and $\mfr R^0 (\Ld_h)\subset \mfr R (\Ld_h) \subset \mfr R^{\mrm{e}} (\Ld_h)$.

In this section we are interested in the mappings 
$\bar \Ph_{p,\al}$ and $\Ps_{\al,\bt}$ as a continuation of \cite{S11b}. 
We will also mention the mapping $\Ld_{q,\al}$ in
Sections 4 and 5.  Their definitions are as follows.  

1. Given $p>0$ [resp.\ $q>0$] and $-\infty<\al<2$, 
let $a_h=0$, $b_h=1$, and $h(u)=\Gm(p)^{-1} (1-u)^{p-1} u^{-\al-1}$ 
[resp.\ $h(u)=\Gm(q)^{-1} (-\log u))^{q-1} u^{-\al-1}$].  
Then $h$ satisfies Condition (C) with $\int_0^1 h(u) u^2du<\infty$ and
$c_h$ is finite and infinite according as $\al<0$ or $\al\ges 0$;
$h^*$ satisfies condition (C) with $\int_1^{\infty} h^*(u)du<\infty$
and hence $c_{h^*}<\infty$.
The mapping $\Ld_h$ is denoted by $\bar \Ph_{p,\al}$ 
[resp.\ $\Ld_{q,\al}$].

2. Given $-\infty<\al<2$ and $\bt>0$, let $a_h=0$, $b_h=\infty$, and $h(u)=
u^{-\al-1} e^{-u^{\bt}}$.  Then $h$ satisfies Condition (C) 
with $\int_0^{\infty} h(u) u^2du<\infty$ and
$c_h$ is finite and infinite according as $\al<0$ or $\al\ges 0$;
$h^*$ satisfies condition (C) with $\int_0^{\infty} h^*(u)du<\infty$
and hence $c_{h^*}<\infty$.
The mapping $\Ld_h$ is denoted by $\Ps_{\al,\bt}$.

Let $\Ld_h$ equal $\bar \Ph_{p,\al}$ or $\Ps_{\al,\bt}$.  
The domains $\mfr D$, 
$\mfr D^0$, $\mfr D^{\mrm e}$ and the ranges $\mfr R$, $\mfr R^0$,
$\mfr R^{\mrm e}$ of both $\Ld_h$ and $\Ld_h^*$ are given description in 
\cite{S11b} if $\al\neq 1$.  In the case $\al=1$ the domains of 
$\Ld_h$ and $\Ld_h^*$ and the ranges $\mfr R(\Ld_h)$, $\mfr R^0(\Ld_h)$,
$\mfr R^{\mrm e}(\Ld_h)$, and $\mfr R^{\mrm e}(\Ld_h^*)$ are described
in \cite{S11b}, but $\mfr R(\Ld_h^*)$ and $\mfr R^0(\Ld_h^*)$ are not 
treated. Now we handle them, using the notion of weak drift.
In the description of $\mfr R^{\mrm e}(\Ld_h)$ and $\mfr R^{\mrm e}(\Ld_h^*)$,
we need to use some notions.  
A L\'evy measure $\nu_{\mu}$ is said to have a radial decomposition (rad.\ dec.) 
$(\ld(d\xi), \nu_{\xi}(dr))$ if $\nu_{\mu}(B)=\int_S \ld(d\xi)\int_{\Rpl}
1_B(r\xi)\nu_{\xi}(dr)$, $B\in\mcal B(\R^d)$, where $\ld$ is a $\sg$-finite 
measure on $S$ with $\ld(S)\ges 0$ and $\{\nu_{\xi}(dr)\cl \xi\in S\}$ is a 
measurable family of $\sg$-finite measures on $\Rpl$ with $\nu_{\xi}(\Rpl)>0$;
the decomposition is unique up to a change to  $(c(\xi)\ld(d\xi), 
c(\xi)^{-1}\nu_{\xi}(dr))$ with a positive, finite, measurable function 
$c(\xi)$ on $S$.  A $[0,\infty]$-valued function $\ph (u)$ on $\Rpl$ 
is said to be monotone of order $p>0$  
if $\ph (u)$ is locally integrable on $\Rpl$ and there is 
a locally finite measure $\sg$ on $\Rpl$ such that
$\ph (u)=\Gm(p)^{-1} \int_{(u,\infty)} (r-u)^{p-1} \sg(dr)$ for $u\in\Rpl$.
A function $\ph (u)$ on $\Rpl$ is said to be completely monotone if it is 
monotone of order $p$ for every $p>0$.  
In Theorem 4.4 and (4.23) of \cite{S11b} it is shown that if $\Ld_h=\bar\Ph_{p,1}$,
then
\begin{align}
\mfr R^{\mrm e}(\Ld_h)&=\{ \mu\in ID\colon \text{$\nu_{\mu}$ has a 
rad.\ dec.\ $(\ld(d\xi),u^{-2}k_{\xi}(u)du)$ such that $k_{\xi}(u)$}\nonumber\\
\label{3.1} &\phantom{XX}\text{is measurable in $(\xi,u)$ and monotone of order $p$ 
in $u\in\Rpl$}\},\\
\mfr R^{\mrm e}(\Ld_h^*)&=\{ \mu\in ID_0\colon \text{$\nu_{\mu}$ has a 
rad.\ dec.\ $(\ld(d\xi),u^{-2}k_{\xi}(u^{-1})du)$ such that $k_{\xi}(v)$}\nonumber\\
\label{3.2} &\phantom{XX}\text{is measurable in $(\xi,v)$ and monotone of 
order $p$ in $v\in\Rpl$}\}.
\end{align}
In Theorem 4.6 of \cite{S11b} it is shown that if $\Ld_h=\Ps_{1,\bt}$, then
\begin{align}
\mfr R^{\mrm e}(\Ld_h)&=\{ \mu\in ID\colon \text{$\nu_{\mu}$ has a 
rad.\ dec.\ $(\ld(d\xi),u^{-2}k_{\xi}(u^{\bt})du)$ such that $k_{\xi}(v)$}\nonumber\\
\label{3.1a} &\phantom{XX}\text{is measurable in $(\xi,v)$ and completely monotone  
in $v\in\Rpl$}\},\\
\mfr R^{\mrm e}(\Ld_h^*)&=\{ \mu\in ID_0\colon \text{$\nu_{\mu}$ has a 
rad.\ dec.\ $(\ld(d\xi),u^{-2}k_{\xi}(u^{-\bt})du)$ such that $k_{\xi}(v)$}\nonumber\\
\label{3.2a} &\phantom{XX}\text{is measurable in $(\xi,v)$ and completely monotone 
in $v\in\Rpl$}\}.
\end{align}

Our result is as follows.

\begin{thm}\label{3t1}
Let $\Ld_h=\bar \Ph_{p,1}$ with $p>0$ or $\Ld_h=\Ps_{1,\bt}$ with $\bt>0$.
Then,
\begin{align}
\label{3t1.1} \mfr R(\Ld_h)&=\{\mu\in \mfr R^{\mrm e}(\Ld_h)\colon 
\text{$\mu$ has weak mean $0$}\},\\
\label{3t1.2} \mfr R^0(\Ld_h)&=\{\mu\in \mfr R^{\mrm e}(\Ld_h)\colon 
\text{$\mu$ has weak mean $0$ absolutely}\},\\
\label{3t1.3} \mfr R(\Ld_h^*)&=\{\mu\in \mfr R^{\mrm e}(\Ld_h^*)\colon 
\text{$\mu$ has weak drift $0$}\},\\
\label{3t1.4} \mfr R^0(\Ld_h^*)&=\{\mu\in \mfr R^{\mrm e}(\Ld_h^*)\colon 
\text{$\mu$ has weak drift $0$ absolutely}\}.
\end{align}
\end{thm}

\begin{proof}
The assertions \eqref{3t1.1} and \eqref{3t1.2} are shown in 
Theorems 4.4 and 4.6 of \cite{S11b}.  In order to obtain 
\eqref{3t1.3} and \eqref{3t1.4} from these, we use the basic relations
of conjugates of stochastic integral mappings with inversions given by 
\[
\mfr R (\Ld_h^*)_0= (\mfr R (\Ld_h)_0)' \quad
\mfr R^{\mrm{e}} (\Ld_h^*)_0= (\mfr R^{\mrm{e}} (\Ld_h)_0)',\quad
\mfr R^0 (\Ld_h^*)_0= (\mfr R^0 (\Ld_h)_0)'
\]
in Theorem 3.6 of \cite{S11b}. We have
\begin{align*}
\mfr R (\Ld_h^*)_0&=\{\mu\in ID_0\colon \mu'\in \mfr R (\Ld_h)_0\}\\
&=\{\mu\in ID_0\colon \mu'\in \mfr R^{\mrm{e}} (\Ld_h)_0 
\text{ and $\mu'$ has weak mean $0$} \}\\
&=\{\mu\in ID_0\colon \mu\in \mfr R^{\mrm{e}} (\Ld_h^*)_0 
\text{ and $\mu$ has weak drift $0$} \}
\end{align*}
from \eqref{3t1.1}, 
since $\mu'$ has weak mean $0$ if and only if $\mu$ has weak 
drift $0$ by virtue of Theorem \ref{2t1}.
This proves \eqref{3t1.3}, since $\mfr R^{\mrm{e}} (\Ld_h^*)_0= 
\mfr R^{\mrm{e}} (\Ld_h^*)$ from  \eqref{3.2} and \eqref{3.2a}. 
The proof of \eqref{3t1.4} is similarly obtained from \eqref{3.2}, 
\eqref{3.2a}, and
\eqref{3t1.2}, using the fact in Theorem \ref{2t1} that
$\mu'$ has weak mean $0$ absolutely if and only if $\mu$ has weak 
drift $0$ absolutely.
\end{proof}

\section{Limits of some nested classes}

For a stochastic integral mapping $\Ph_f$ its iterations $\Ph_f^n$, 
$n=1,2,\ldots$, are defined as $\Ph_f^1=\Ph_f$ and $\Ph_f^{n+1} \rh=
\Ph_f(\Ph_f^n \rh)$ with
$\mfr D(\Ph_f^{n+1})=\{\rh\in \mfr D(\Ph_f^n) \cl \Ph_f^n \rh\in \mfr D(\Ph_f)\}$.  
Then we get nested classes 
$ID\supset \mfr R(\Ph_f)\supset \mfr R(\Ph_f^2)\supset \cdots$.  Let  
$\mfr R_{\infty}(\Ph_f)=\bigcap_{n=1}^{\infty} \mfr R(\Ph_f^n)$, the limit of the 
nested classes.  The class $\mfr R_{\infty}(\Ph_f)$ is possibly identical for 
different functions $f$. 
For example, it was shown in \cite{MS09} that $\mfr R_{\infty}(\Ph_f)$ equals
the class $L_{\infty}$ for many stochastic integral mappings
$\Ph_f$ known at that time.  Here $L_{\infty}$ is the class of completely selfdecomposable
distributions on $\R^d$, which is the smallest class closed under convolution
and weak convergence and containing all stable distributions on $\R^d$. 
A distribution $\mu
\in ID$ belongs to $L_{\infty}$ if and only if  
\begin{equation*}
\nu_{\mu}(B)=\int_{(0,2)} \Gm_{\mu}(d\bt)\int_S \ld_{\bt}^{\mu}(d\xi) \int_0^{\infty}
1_B(r\xi) r^{-\bt-1} dr,\quad B\in\mcal B(\R^d),
\end{equation*}
where $\Gm_{\mu}$ is a measure on 
$(0,2)$ satisfying $\int_{(0,2)} (\bt^{-1}+(2-\bt)^{-1}) \Gm_{\mu}(d\bt) <\infty$
and $\{\ld_{\bt}^{\mu}\cl \bt\in(0,2)\}$ is a measurable family of probability
measures on $S$.  This representation of $\nu_{\mu}$ is unique.
For a Borel subset $E$ of $(0,2)$, $L_{\infty}^E$ denotes
the class of $\mu\in L_{\infty}$ such that $\Gm_{\mu}$ is concentrated on $E$.

We are interested in what classes appear as $\mfr R_{\infty}(\Ld_h)$ and 
$\mfr R_{\infty}(\Ld_h^*)$ for stochastic integral mappings $\Ld_h$ associated
with functions $h$ satisfying Condition (C).  In \cite{S11a,S11b} the 
description of $\mfr R_{\infty}(\Ld_h)$ and 
$\mfr R_{\infty}(\Ld_h^*)$ is given for $\Ld_h$ equal to $\bar\Ph_{p,\al}$, 
$\Ld_{q,\al}$, and $\Ps_{\al,1}$ with $\al\in(-\infty,1)\cup(1,2)$, $p\ges 1$, 
and $q>0$.    
The  description of $\mfr R_{\infty}(\Ld_h)$ is also given in
the case $\al=1$, $p\ges 1$, and $q=1$ in \cite{S11a}.
Actually $\Ld_{1,\al}=\bar\Ph_{1,\al}$. 
Now let us treat $\mfr R_{\infty}(\Ld_h^*)$ for $\Ld_h$ equal to $\bar\Ph_{p,1}$ 
and $\Ps_{1,1}$ with $p\ges 1$. Again the notion of weak drift is crucial.

\begin{thm}\label{3t2}
Let $\Ld_h=\bar \Ph_{p,1}$ with $p\ges 1$ or $\Ld_h=\Ps_{1,1}$.  Then 
\begin{align}
\label{3t2.0} \mfr R_{\infty}(\Ld_h)&=
L_{\infty}^{(1,2)} \cap\{\mu\in ID\cl\text{$\mu$ has weak mean $0$} \},\\
\label{3t2.1}
\mfr R_{\infty}(\Ld_h^*)&=(L_{\infty}^{(0,1)})_0 \cap\{\mu\in ID\cl
\text{$\mu$ has weak drift $0$} \}.
\end{align}
\end{thm}

\begin{proof}
The description \eqref{3t2.0} of $\mfr R_{\infty}(\Ld_h)$ is
shown in Theorem 1.1 of \cite{S11a}. 
We have $\mfr R_{\infty}(\Ld_h^*)_0=(\mfr R_{\infty}(\Ld_h)_0)'$ in
Theorem 6.3 of \cite{S11b}, and $((L_{\infty}^{(1,2)})_0)'=
(L_{\infty}^{(0,1)})_0$ obtained from Proposition 6.1 of  \cite{S11b}.
Hence $\mfr R_{\infty}(\Ld_h^*)_0$ is identical with the right-hand
side of \eqref{3t2} by virtue of Theorem \ref{2t1}.  It remains to 
see $\mfr R_{\infty}(\Ld_h^*)=
\mfr R_{\infty}(\Ld_h^*)_0$.  But, since $f_{h^*}(s)\asymp s^{-1}$ as
$s\dar 0$ by (4.3) and (4.38) of \cite{S11b},
we have $\int_0^{c_{h^*}} f_{h^*} (s)^2 ds=\infty$, and hence 
$\mfr D(\Ld_h^*) \subset ID_0$ and  $\mfr R(\Ld_h^*) \subset ID_0$ 
as in Proposition 3.8 of \cite{S11b}. 
Now it follows that $\mfr R_{\infty}(\Ld_h^*)\subset ID_0$.
\end{proof}
 
\section{Weak law of large numbers and weak version of\\
Shtatland's theorem}

Shtatland \cite{Sh65} proves that, for a L\'evy process $\{X_t^{(\mu)}\cl
t\ges 0\}$ on $\R$,  $\lim_{\ep\dar 0} \ep^{-1}X_{\ep}^{(\mu)}\break =c$ 
almost surely for $c\in\R$  if and only if $\mu$ has drift $c$ and no
Gaussian part.  The \lq\lq if" part is easily extended to $\R^d$; 
see Theorem 43.20 of \cite{S99}.  
A weaker conclusion is that, for a L\'evy process $\{X_t^{(\mu)}\cl
t\ges 0\}$ on $\R^d$, if $\mu\in ID_0$ and $\mu$ has drift $c$,
then $\mcal L(\ep^{-1} X_{\ep}^{(\mu)})\to \dl_c$ as 
$\ep\dar 0$.
The following fact shows its connection with the weak law of large
numbers through inversion.

\begin{thm}\label{4p1}
Let $\mu\in ID_0$ and $c\in\R^d$.  Then $\mcal L(\ep^{-1} X_{\ep}^{(\mu)})
\to \dl_c$ as $\ep\dar 0$ if and only if $\mcal L(t^{-1} X_t^{(\mu')})
\to \dl_{-c}$ as $t\to \infty$.
\end{thm}

\begin{proof}
The core of the proof is the formula $(T_b\mu)'=(T_{b^{-1}}(\mu'))^{b^2}$ 
for $b>0$ and $\mu\in ID_0$ proved in Proposition 2.4 of \cite{S11b},
where $T_b$ is the dilation $(T_b\mu)(B)=\int_{\R^d} 1_B(bx)\mu(dx)$
and $\mu^t=\law (X_t^{(\mu)})$.  Notice that $T_b(\mu^t)=(T_b \mu)^t$.
We also use properties of the inversion in 
Proposition 2.1 (vi), (viii), and (ix) of \cite{S11b}.
Assume that $\mcal L(\ep^{-1} X_{\ep}^{(\mu)})
\to \dl_c$ as $\ep\dar 0$.  Then $(\mcal L(\ep^{-1} X_{\ep}^{(\mu)}))'
\to \dl_c'=\dl_{-c}$.  We have
\[
(\mcal L(\ep^{-1} X_{\ep}^{(\mu)}))'=(T_{\ep^{-1}}(\mu^{\ep}))'
=(T_{\ep}((\mu^{\ep})'))^{\ep^{-2}} =(T_{\ep}((\mu')^{\ep}))^{\ep^{-2}}
=(T_{\ep}(\mu'))^{\ep^{-1}},
\]
which is equal to $\mcal L(t^{-1}X_t^{(\mu')})$ for $t=\ep^{-1}$.
The converse is similar.
\end{proof}

Necessary and sufficient conditions for the weak law of large numbers
for L\'evy processes are as follows.

\begin{thm}\label{4t1}
Let $\mu\in ID$ and $c\in\R^d$.  The following three statements are equivalent.
\begin{enumerate}
\item The L\'evy process $\{X_t^{(\mu)}\cl
t\ges 0\}$ on $\R^d$ satisfies $\mcal L(t^{-1} X_t^{(\mu)})
\to \dl_c$ as $t\to \infty$.
\item The distribution $\mu$ has weak mean $c$ and
\begin{equation}\label{4t1.1}
\lim_{t\to\infty} t\int_{|x|>t} \nu_{\mu}(dx)=0.
\end{equation}
\item The distribution $\mu$ satisfies
\begin{gather}
\label{4t1.6} \lim_{t\to\infty} \int_{|x|\les t} x\mu(dx)=c,\\
\label{4t1.7} \lim_{t\to\infty} t\int_{|x|>t} \mu(dx)=0.
\end{gather}
\end{enumerate}
\end{thm}

\begin{proof}
The equivalence of (i) and (ii) is as follows. 
It is convenient to use the L\'evy--Khintchine representation in the
form
\begin{equation*}
\wh\mu(z)=\exp\Bigl[ -\tfrac12 \langle z,A_{\mu} z\rangle +
\int_{\R^d}(e^{i\la z,x\ra}-1-i\la z,x\ra c(x))
\nu_{\mu}(dx)+i\la\gm^{\sharp}_{\mu},z\ra\Big]
\end{equation*}
for $\mu\in ID$ 
with $c(x)=1_{\{|x|\les1\}}(x)+|x|^{-1} 1_{\{|x|>1\}}(x)$ adopted by
Rajput and Rosinski \cite{RR89} and Kwapie\'n and Woyczy\'nski \cite{KW92}, 
as in (2.4) of \cite{S11b}. Let us call $\gm^{\sharp}_{\mu}$ the 
$\sharp$-location parameter of $\mu$; $\gm^{\sharp}_{\mu}$ is related to
the location parameter $\gm_{\mu}$ in \eqref{1.1} as
\begin{equation}\label{4t1.2}
\gm_{\mu}^{\sharp}=\gm_{\mu}+\int_{|x|>1} |x|^{-1}x\nu_{\mu}(dx).
\end{equation}
The dilation $T_b\mu$ of $\mu$ with $0<b<1$ has triplet
\[
A_{T_b\mu}=b^2 A_{\mu},\quad \nu_{T_b\mu}=T_b \nu_{\mu},\quad
\gm_{T_b\mu}=b\gm_{\mu}+b\int_{1<|x|\les b^{-1}} x\nu_{\mu}(dx)
\]
as in (2.7) of \cite{S11b}.  Hence, for $t>1$, $\mcal L(t^{-1}X_t^{(\mu)})
=T_{t^{-1}}(\mu^t)=(T_{t^{-1}}\mu)^t$ has Gaussian covariance matrix
$t^{-1}A_{\mu}$, L\'evy measure $t T_{t^{-1}}\nu_{\mu}$, and $\sharp$-location 
parameter 
\[
\gm_{\mu}+\int_{1<|x|\les t} x\nu_{\mu}(dx)+t\int_{|x|>t} |x|^{-1}
x\nu_{\mu}(dx)
\]
from \eqref{4t1.2}.  Since $c(x)$ is continuous on $\R^d$, 
we can use Theorem 8.7 of \cite{S99}
and see that $\mcal L(t^{-1} X_t^{(\mu)})
\to \dl_c$ as $t\to \infty$ if and only if
\begin{align}
\label{4t1.3} &\text{$t\nu_{\mu}(tB)\to 0$ for all $B\in\mcal B(\R^d)$ such
that $0$ is not in the closure of $B$},\\
\label{4t1.4} &\lim_{\et\dar0}\limsup_{t\to\infty} \Bigl( \la z,t^{-1}A_{\mu}
z\ra + t\int_{|x|<\et}\la z,t^{-1}x\ra^2 \nu_{\mu}(dx)\Bigr) =0\text{ for }
z\in\R^d,\\
\intertext{and}
\label{4t1.5} &\gm_{\mu}+\int_{1<|x|\les t} x\nu_{\mu}(dx)+t\int_{|x|>t} |x|^{-1}
x\nu_{\mu}(dx)\to c.
\end{align}
Condition \eqref{4t1.3} is the same as $t\int_{|x|>t\et} \nu_{\mu}(dx)\to 0$
for $\et>0$, which is equivalent to \eqref{4t1.1}.  Condition \eqref{4t1.4} is
always satisfied.  If \eqref{4t1.1} holds, then $t\int_{|x|>t} |x|^{-1}
x\nu_{\mu}(dx)\to 0$ and condition \eqref{4t1.5} is expressed as $\mu$ has
weak mean $c$.

Next, let us prove the equivalence of (i) and (iii). If (i) holds, then
$n^{-1}X_n^{(\mu)}\to c$ in probability as $n=1,2,\ldots\to\infty$.
Since $\{X_n^{(\mu)}\}$ is the sum of i.\,i.\,d.\ random variables, we see 
from the theorem in p.\,565 of Feller \cite{F71} and Theorem 36.4 of \cite{S99}
that (i) implies (iii).  These theorems also show that (iii) implies that
$n^{-1}X_n^{(\mu)}\to c$ in probability as $n\to\infty$.  Statement (i)
follows from this, since, for $n\les t<n+1$,
\begin{gather*}
t^{-1} X_t^{(\mu)}=t^{-1} ( X_t^{(\mu)}- X_n^{(\mu)}) +(t^{-1}-n^{-1}) X_n^{(\mu)}
+n^{-1}X_n^{(\mu)},\\
t^{-1} |X_t^{(\mu)}- X_n^{(\mu)}|\overset{\mrm{law}}{=} t^{-1} |X_{t-n}^{(\mu)}|
\les t^{-1}{\ts \sup_{s\les1}} |X_s^{(\mu)}|\to 0\quad\text{a.\,s.,}\quad t\to\infty,\\
\intertext{and}
|(t^{-1}-n^{-1}) X_n^{(\mu)}|\les n^{-2} |X_n^{(\mu)}|\to 0
\quad\text{in probability,}\quad t\to\infty.
\end{gather*}
This finishes the proof.
\end{proof}

As to the inversion version of Theorem \ref{4t1}, we can prove the 
equivalence of the analogues of (i) and (ii).

\begin{thm}\label{4t2}
Let $\mu\in ID_0$ and $c\in\R^d$.  The following two statements are equivalent.
\begin{enumerate}
\item The L\'evy process $\{X_t^{(\mu)}\cl
t\ges 0\}$ on $\R^d$ satisfies $\mcal L(\ep^{-1} X_{\ep}^{(\mu)})
\to \dl_c$ as $\ep\dar 0$.
\item  The distribution $\mu$ has weak drift $c$ and
\begin{equation}\label{4t2.1}
\lim_{\ep\dar 0} \ep^{-1}\int_{|x|\les\ep} |x|^2\nu_{\mu}(dx)=0.
\end{equation}
\end{enumerate}
\end{thm}

\begin{proof}
Combine Theorem \ref{4p1} with the equivalence of (i) and (ii) of
Theorem \ref{4t1}.  Then,
$\mcal L(\ep^{-1} X_{\ep}^{(\mu)})
\to \dl_c$ as $\ep\dar 0$ if and only if $\mu'$ has weak mean $-c$ and
$t\int_{|x|>t} \nu_{\mu'}(dx)\break \to 0$ as $t\to\infty$.
Use Theorem \ref{2t1} and that 
$t\int_{|x|>t} \nu_{\mu'}(dx)=t\int_{|x|< t^{-1}} |x|^2 \nu_{\mu} (dx)$.
Now we see that our assertion is true.  
\end{proof}

We give two final remarks concerning the conditions \eqref{4t1.1} and
\eqref{4t2.1}.

1. If $\mu\in ID$ has mean, then $\mu$ satisfies 
\eqref{4t1.1}, since $t\int_{|x|>t} \nu_{\mu}(dx)\les
 \int_{|x|>t}|x|\nu_{\mu}(dx)$. If $\mu\in ID$ has drift, then $\mu$ satisfies 
\eqref{4t2.1}, since $\ep^{-1}\int_{|x|\les\ep} |x|^2\nu_{\mu}(dx)\les
 \int_{|x|\les\ep}|x|\nu_{\mu}(dx)$.
 
2.  Let $\Ld_h$ be one of $\bar\Ph_{p,1}$
with $p\ges1$, $\Ld_{q,1}$ with $q\ges1$, and $\Ps_{1,\bt}$ with $\bt>0$.
Then any $\mu\in \mfr R^{\mrm{e}}(\Ld_h)$ satisfies \eqref{4t1.1} and
any $\mu\in \mfr R^{\mrm{e}}(\Ld_h^*)$ satisfies \eqref{4t2.1}.
To see this, first note that if $\mu\in \mfr R^{\mrm{e}}(\Ld_h)$ 
[resp.\ $\mfr R^{\mrm{e}}(\Ld_h^*)$] and if $\mu=\mu_0*\mu_1$ with 
$\mu_0\in ID_0$ and $\mu_1$ being Gaussian, then $\mu_0\in 
\mfr R^{\mrm{e}}(\Ld_h)_0$ [resp.\ $\mfr R^{\mrm{e}}(\Ld_h^*)_0$]; 
see Proposition 3.18 of \cite{S10}. Then, for $\mfr R^{\mrm{e}}(\Ld_h)$,
note that any function monotone of order $p\ges1$ is decreasing to $0$ 
(Corollary 2.6 of \cite{S10}) 
and use Lemma 4.2 of \cite{MU10}.  For $\mfr R^{\mrm{e}}(\Ld_h^*)$,
use the following analogue of Lemma 4.2 of \cite{MU10}:  {\it Under the
assumption that $\mu\in ID$ is such that $\nu_{\mu}$ has a 
rad.\ dec.\ $(\ld(d\xi),u^{-2}l_{\xi}(u)du)$ with $l_{\xi}(u)$
measurable in $(\xi,u)$ and increasing in $u\in\Rpl$, we have 
$l_{\xi}(0+)=0$ for $\ld$-a.\,e.\ $\xi$ if and only if \eqref{4t2.1} holds}.
This is because 
\[
\ep^{-1}\int_{|x|\les\ep}|x|^2 \nu_{\mu}(dx)
=\int_S \ld(d\xi)\int_{(0,1]}v^2 v^{-2} l_{\xi}(\ep v)dv.
\]
Now let $\mu\in \mfr R^{\mrm{e}}(\Ld_h^*)$.  In order to prove that $\mu$
satisfies \eqref{4t2.1}, let us show that $\nu_{\mu}$ has a rad.\ dec.\ 
$(\ld(d\xi),u^{-2}l_{\xi}(u)du)$ with $l_{\xi}(u)$
measurable in $(\xi,u)$, increasing in $u\in\Rpl$, and  
$l_{\xi}(0+)=0$ for $\ld$-a.\,e.\ $\xi$.  
Indeed, if $\Ld_h=\bar\Ph_{p,1}$ 
with $p\ges1$, then $\nu_{\mu}$ has a rad.\ dec.\ 
$(\ld(d\xi),u^{-p-1}k_{\xi}(u)du)$ with $k_{\xi}(u)$ increasing of
order $p$ on $\Rpl$ by Theorem 4.4 of \cite{S11b} and $l_{\xi}(u)
=u^{1-p}k_{\xi}(u)$ is increasing in $u$ and $l_{\xi}(0+)=0$, since
$u^{p-1} k_{\xi}(u^{-1})$ is monotone of order $p$ in $u\in\Rpl$ by
Proposition 4.3 of \cite{S11b}.
If $\Ld_h=\Ld_{q,1}$ with $q\ges1$, then $\nu_{\mu}$ 
has a rad.\ dec.\ $(\ld(d\xi),u^{-2}h_{\xi}(\log u)du)$ with $h_{\xi}(y)$ 
being increasing of order $q$ in $y\in \R$ by Theorem 4.5 of \cite{S11b}
and hence $h_{\xi}(y)$ is increasing and tends to $0$ as $y\to-\infty$ 
by Proposition 4.3 of \cite{S11b}.
If $\Ld_h=\Ps_{1,\bt}$ with $\bt>0$, then $\nu_{\mu}$ 
has a rad.\ dec.\ $(\ld(d\xi),u^{-2}k_{\xi}(u^{-\bt})du)$ with $k_{\xi}(v)$ 
completely monotone in $v\in\Rpl$ by Theorem 4.6 of \cite{S11b}
and thus $l_{\xi}(u)=k_{\xi}(u^{-\bt})$ is increasing in $u\in\Rpl$ and 
$l_{\xi}(0+)=0$.

\medskip
\noindent {\it Acknowledgments}.   The authors thank an anonymous 
referee for giving them valuable advice on improvement of the paper.

\end{document}